\documentclass[a4paper]{article}

\usepackage[T1]{fontenc} 
\usepackage[utf8]{inputenc}
\usepackage[english]{babel}
\usepackage{alphabeta}
\usepackage{authblk}
\usepackage[frozencache=true,cachedir=minted-cache]{minted}

\usepackage{amsmath, amsthm, amsfonts, amssymb}
\usepackage{enumerate}
\usepackage{hyperref}
\usepackage{graphicx}
\usepackage[all]{xy}
\usepackage{tikz-cd}
\usepackage{cancel}
\usepackage{orcidlink}
\usetikzlibrary{graphs,decorations.pathmorphing,decorations.markings, arrows.meta}

\newtheorem{thm}{Theorem}[section]
\newtheorem{cor}[thm]{Corollary}
\newtheorem{lem}[thm]{Lemma}
\newtheorem{prop}[thm]{Proposition}
\newtheorem{exam}[thm]{Example}
\theoremstyle{definition}
\newtheorem{defi}[thm]{Definition}
\theoremstyle{remark}
\newtheorem{rem}[thm]{Remark}
\newtheorem{nota}[thm]{Notation}

\newcommand*{\cat}[1]{\mathcal{#1}}
\newcommand*{\set}{\operatorname*{set}}
\newcommand*{\obj}[1]{#1_0}

\newcommand*{\Hom}[3]{\operatorname{Hom}_{\cat{#1}} \left(#2, #3\right)}

\title{The Burnside ring of simple $\cat{C}$-sets}

\author[1]{J. Miguel Calder\'on \orcidlink{0009-0007-4685-8837}}
\author[2]{Alberto G. Raggi-C\'ardenas \orcidlink{0000-0003-1720-1733}}
\author[3]{Itzel Rosas \orcidlink{0009-0002-3411-8456}}
\author[4]{Ram\'on H. Ruiz-Medina \orcidlink{0000-0003-2916-9160}}
\date{}

\affil[1]{Centro de Investigaci\'on en Matem\'aticas, A.C., Unidad M\'erida,\\
Parque Cient\'ifico y Tecnol\'ogico de Yucat\'an,\\
Carretera Sierra Papacal--Chuburna Puerto Km 5.5,\\
Sierra Papacal, M\'erida, YUC 97302,\\
M\'exico. \\ \texttt{calderonl@cimat.mx}}
\affil[2]{Centro de Ciencias Matem\'aticas, UNAM, Morelia, Michoac\'an 58089 \\ \texttt{agraggi@gmail.com}}
\affil[3]{PCCM, UNAM -- UMSNH, Morelia, Michoac\'an \\ \texttt{irosas@matmor.unam.mx}}
\affil[4]{Centro de Ciencias Matem\'aticas, UNAM, Morelia, Michoac\'an 58089 \\ \texttt{harath@matmor.unam.mx}}

\begin{document}
\maketitle


\begin{abstract}
The Burnside ring of a finite category, introduced by Webb, generalizes the classical Burnside ring of a finite group. However, unlike the classical case, the Burnside ring of a finite category has finite rank if and only if the category is equivalent to a groupoid.
In this article, we introduce a new invariant associated with a finite category $\cat{C}$, called the \emph{simple Burnside ring} of $\cat{C}$ and denoted by $B^S(\cat{C})$. This construction is obtained from simple $\cat{C}$-sets and generalizes the classical Burnside ring of a finite group. Moreover, the ring $B^S(\cat{C})$ always has finite rank.

We develop the basic theory of simple $\cat{C}$-sets and study several structural properties of the ring $B^S(\cat{C})$. In particular, we determine all ring homomorphisms from $B^S(\cat{C})$ to $\mathbb{Z}$, describe its prime spectrum, and obtain a decomposition theorem expressing $B^S(\cat{C})$ as a product of simple Burnside rings of strongly connected subcategories.
\end{abstract}

\section{Introduction}
The Burnside ring is one of the main tools in the study of representations of finite groups. It was introduced implicitly by Burnside in~\cite{burnside1909theory}, and later developed by Dress in~\cite{dress1969characterisation}. This ring is an algebraic invariant associated with a finite group and has been extensively studied by several mathematicians throughout history. Dress and  Schwänzl computed the spectrum of the Burnside ring in \cite{dress1969characterisation}  and ~\cite{spec-SCHW} respectively. while Bouc studied its idempotents and ideals in~\cite{bouc2010}. The relation between the Burnside ring and the ghost ring was studied in~\cite{dress1992application}. Bouc systematically developed and presented the classical theory of Burnside rings in \cite{bouc2000burnside}. Since then, the theory has been extended in several directions. For instance, the Burnside ring of a monoid was introduced in \cite{weissmann2025burnside}.

Recently,  Webb generalized the notion of the Burnside ring to finite categories in~\cite{Webb23}. Given a finite category $\cat{C}$, Webb considers the category whose objects are covariant functors
\[
F \colon \cat{C} \longrightarrow \mathbf{Set},
\]
called $\cat{C}$-sets, and whose morphisms are natural transformations between them. When $\cat{C}$ is the delooping category of a finite group $G$, these functors correspond precisely to finite $G$-sets. The Grothendieck group of this category is called the Burnside ring of the category $\cat{C}$ and is denoted by $B(\cat{C})$. In this case, the construction coincides with the classical Burnside ring of the finite group $G$.

The study of this generalized Burnside ring presents new opportunities and challenges.In~\cite{semisimple}, it was proved that the Burnside ring of a finite category has finite rank if and only if the category is equivalent to a groupoid.
In this article, we introduce a new invariant associated with a category $\cat{C}$, called the \emph{simple Burnside ring} of $\cat{C}$, denoted by $B^S(\cat{C})$. This construction is also a generalization of the Burnside ring of a finite group $G$, and moreover, it always has finite rank.

In Section \ref{Preliminares}, we introduce the basic definitions and preliminary results concerning the theory of $\cat{C}$-sets that will be needed throughout the article. In Section \ref{simple y semisimple}, we define the notion of a simple $\cat{C}$-set, namely a $\cat{C}$-set generated by any element of any of its evaluations. We also introduce the \emph{simple Burnside ring} of $\cat{C}$. In Section \ref{marks}, we define ring homomorphisms, called \emph{marks}, from $B^S(\cat{C})$ to $\mathbb{Z}$. These homomorphisms allow us to determine all ring homomorphisms from $B^S(\cat{C})$ to $\mathbb{Z}$. In Section \ref{spec}, we determine the spectrum of the ring $B^S(\cat{C})$. In Section \ref{descompositon}, we obtain a decomposition of the ring $B^S(\cat{C})$ as a product of simple Burnside rings of certain subcategories called \emph{strongly connected} subcategories. Finally, in Section \ref{strongly connected}, we study the simple Burnside ring of a strongly connected category.

\section{Preliminaries}\label{Preliminares}

Throughout this article, all categories are assumed to be finite. Let $\cat{C}$ be a category. We denote the set of objects of $\cat{C}$ by $\cat{C}_0$, a $\cat{C}$-set is a covariant functor from $\cat{C}$ to the category of finite sets, denoted by $\mathbf{set}$. This notion was introduced by  Webb in~\cite{Webb23}.
\begin{exam}
Let $\cat{C}$ be a category and let $x \in \cat{C}_0$. We define the $\cat{C}$-set $Hom_{\cat{C}}(x,-)$ as follows:
\[
    Hom_{\cat{C}}(x,-)(y) := Hom_{\cat{C}}(x,y), \quad \text{for all } y \in \cat{C}_0,
\]
and for every morphism $f \colon y \to z$ in $\cat{C}$, we define
\begin{align*}
     Hom_{\cat{C}}(x,-)(f) \colon Hom_{\cat{C}}(x,y) &\to Hom_{\cat{C}}(x,z)\\
     &g \mapsto f \circ g.
\end{align*}

\end{exam}

Given a $\cat{C}$-set $\Omega$, we say that a morphism $\alpha \in \Hom{C}{x}{y}$ acts on an element $a \in \Omega(x)$ via the evaluation of its image under $\Omega$, and, when the $\cat{C}$-set is clear in the context, we simply denote $\alpha \cdot a = \Omega(\alpha)(a)$.\\

Given a group $G$, the \emph{delooping} (or single-object) category $\cdot_G$ is the one with a single object, and whose morphisms and composition are given by the group elements and the group operation, respectively. A $\cdot_G$-set is equivalent to a group action of $G$ on the set image of the object under the functor.\\

The \emph{empty} $\cat{C}$-set is the functor $\emptyset$ that assigns to each object of $\cat{C}$ the empty set and to each morphism the empty function. A $\cat{C}$-subset of a $\cat{C}$-set $\Omega$ is a subfunctor of $\Omega$. In particular, the empty $\cat{C}$-set is a $\cat{C}$-subset of every $\cat{C}$-set.

A $\cat{C}$-set is called \emph{simple} if it has no proper $\cat{C}$-subsets; equivalently, its only $\cat{C}$-subsets are itself and the empty $\cat{C}$-set.

Let $\Omega$ and $\Lambda$ be $\cat{C}$-sets. Their disjoint union is defined as follows.

\begin{itemize}
    \item On objects: for every object $x\in\obj{C}$,
    \[
    (\Omega\sqcup\Lambda)(x)=\Omega(x)\sqcup\Lambda(x).
    \]

    \item On morphisms: for every morphism $\alpha\in\Hom{C}{x}{y}$,
    \[
    (\Omega\sqcup\Lambda)(\alpha)(t)=
    \begin{cases}
        \Omega(\alpha)(t), & \text{if } t\in\Omega(x),\\
        \Lambda(\alpha)(t), & \text{if } t\in\Lambda(x).
    \end{cases}
    \]
\end{itemize}

Assume that, for every morphism $\alpha\in\Hom{C}{x}{y}$,
\[
\Omega(\alpha)|_{\Omega(x)\cap\Lambda(x)}
=
\Lambda(\alpha)|_{\Omega(x)\cap\Lambda(x)}.
\]
Then the intersection $\Omega\cap\Lambda$ is well defined and is given as follows.

\begin{itemize}
    \item On objects: for every object $x\in\obj{C}$,
    \[
    (\Omega\cap\Lambda)(x)=\Omega(x)\cap\Lambda(x).
    \]

    \item On morphisms: for every morphism $\alpha\in\Hom{C}{x}{y}$,
    \[
    (\Omega\cap\Lambda)(\alpha)
    =
    \Omega(\alpha)|_{(\Omega\cap\Lambda)(x)}.
    \]
\end{itemize} 

\begin{rem}
   For any $\cat{C}$-sets $\Omega_1$ and $\Omega_2$, we have, 
$ 
\Omega_i \subseteq \Omega_1 \sqcup \Omega_2,
$ $i=1,2.
$ 
Moreover, if the intersection $\Omega_1 \cap \Omega_2$ is defined, then
$ \Omega_1 \cap \Omega_2 \subseteq \Omega_i,
$ $i=1,2.$
\end{rem}

Note that if $S_1$ and $S_2$ are distinct simple $\cat{C}$-subsets of a $\cat{C}$-set $\Omega$, then $S_1(x) \cap S_2(x) = \emptyset$ for every object $x \in \obj{C}$. This is a consequence of the fact that the intersection $S_1 \cap S_2$ would be a proper $\cat{C}$-subset of both $S_1$ and $S_2$, contradicting their simplicity. However, it is possible for two simple $\cat{C}$-subsets of any given $\cat{C}$-set to be isomorphic to each other.

A $\cat{C}$-set is said to be \emph{finite} if $\Omega(x)$ is a finite set for every object $x \in \obj{C}$.We say that a $\cat{C}$-set $\Omega$ is \textit{non-empty} if exists $c \in \obj{\cat{C}}$ such that $\Omega(c) \neq \emptyset$. 
\begin{defi}
	A non-empty $\cat{C}$-set $\Omega$ is said to be \textit{indecomposable} if it cannot be expressed properly as a disjoint union, that is, if $\Omega_1, \Omega_2$ are $\cat{C}$-sets such that $\Omega = \Omega_1 \sqcup \Omega_2$, then $\Omega_1 = \varnothing$ or $\Omega_2 = \varnothing$.
\end{defi}

The following is an important result which tells that every finite $\cat{C}$-set has a decomposition in indecomposable $\cat{C}$-sets, this result can be explored in \cite{Webb23}. 

\begin{lem}[Proposition 2.3 in \cite{Webb23}]
Let $\Omega$ be a $\cat{C}$-set. Then there exist indecomposable $\cat{C}$-sets $\Omega_1,\ldots,\Omega_k$ such that
\[
\Omega=\Omega_1\sqcup\Omega_2\sqcup\cdots\sqcup\Omega_k.
\]
\end{lem}

It follows from the definitions that every simple $\cat{C}$-set is indecomposable. It holds that in groupoids, a $\cat{C}$-set is indecomposable if and only if it is transitive.
We denote by $\cat{C}$-set the category whose objects are $\cat{C}$-sets and whose morphisms are natural transformations between them. 
\begin{defi}[Definition 2.8 in \cite{Webb23} ]
	We define the \textit{Burnside ring} of the category $\cat{C}$, denoted by $B(\cat{C})$, as the Grothendieck group of the category $\cat{C}$-set with respect to the disjoint union of $\cat{C}$-sets.
\end{defi}

	We have a product operation in $\cat{C}$-set: given two $\cat{C}$-sets $\Omega, \Psi$, we define $(\Omega \times \Psi)(x) = \Omega(x) \times \Psi(x)$ for every object $x$ of $\cat{C}$, and $(\Omega \times \Psi)(\alpha) = (\Omega(\alpha), \Psi(\alpha))$ for every morphism $\alpha$ of $\cat{C}$. With this definition, the multiplication in $B(C)$ is given by $\times$ on the generators.
The functor $1 _\cat{C}\colon \cat{C} \to \set$ defined by $1_\cat{C}(x) = \{\, \cdot \,\}$ for all $x \in \cat{C}_0$ is the unit object of $B(\cat{C})$.

\begin{prop}\label{mark}
Let $S$ be a  indecomposable  $\cat{C}$-set. The \emph{mark} of $S$ is the function
\begin{align*}
    \varphi_S \colon B(\cat{C}) &\longrightarrow \mathbb{Z},\\
    [\Omega] &\longmapsto |Hom_{\cat{C}\text{-set}}(S,\Omega)|.
\end{align*}
This function is a unital ring homomorphism.
\end{prop}
\begin{proof}
Note that $1_{\cat{C}}(x)=\{ \cdot \}$ for all $x\in \cat{C}_0$. Hence,
\[
   \varphi_S(1_{\cat{C}})= |Hom_{\cat{C}\text{-set}}(S,1_{\cat{C}})| = 1,
\]
so $\varphi_S$ preserves the identity.

Let $R$ and $T$ be $\cat{C}$-sets. 
Let $R$ and $T$ be simple $\cat{C}$-sets. We first prove additivity. We claim that
\[
    Hom_{\cat{C}\text{-set}}(S, T \sqcup R) \cong Hom_{\cat{C}\text{-set}}(S,T) \sqcup Hom_{\cat{C}\text{-set}}(S,R).
\]

Indeed, let $\eta \colon S \to T \sqcup R$ be a natural transformation. Then
\[
    S = \eta^{-1}(T) \sqcup \eta^{-1}(R).
\]
Since $S$ is decomposable, either $\eta^{-1}(T) = \emptyset$ or $\eta^{-1}(R) = \emptyset$. In the first case, $\eta$ factors through $R$, and in the second case, it factors through $T$. This proves the claim. Hence,
\begin{align*}
    \varphi_S([T] + [R])
    &= |Hom_{\cat{C}\text{-set}}(S, T \sqcup R)| \\
    &= |Hom_{\cat{C}\text{-set}}(S,T)| + |Hom_{\cat{C}\text{-set}}(S,R)| \\
    &= \varphi_S([T]) + \varphi_S([R]).
\end{align*}

We now prove multiplicativity. By definition,
$[R]\cdot[T ]=[R\times T]$,  then
\begin{align*}
    \varphi_S([R]\cdot[T ])=\varphi_S([R\times T])=|Hom_{\cat{C}\text{-set}}(S,R\times T)|
\end{align*}
By the universal property of $\times$, we have $$Hom_{\cat{C}\text{-set}}(S,R\times T)\cong Hom_{\cat{C}\text{-set}}(S,R)\times Hom_{\cat{C}\text{-set}}(S,T).$$ Then $|Hom_{\cat{C}\text{-set}}(S,R\times T)|=|Hom_{\cat{C}\text{-set}}(S,R)|\cdot |Hom_{\cat{C}\text{-set}}(S,T)|=\varphi_S([R])\cdot \varphi_S([T]).$
    
\end{proof}

\section{Simple $\cat{C}$-sets}\label{simple y semisimple}
In this section, we study simple $\cat{C}$-sets, which can be viewed as a generalization of transitive $G$-sets (see~\cite{semisimple}). We also prove that there exist only finitely many simple $\cat{C}$-sets up to isomorphism.
\begin{defi}
Let $\Omega$ be a $\cat{C}$-set and let $x \in \cat{C}_0$. For an element $a \in \Omega(x)$, we define the $\cat{C}$-set \emph{generated by $a$ in $\Omega$}, denoted by $\cat{C}^\Omega \cdot a$, as follows:
\begin{itemize}
    \item On objects: for each $y \in \obj{C}$,
    \[
    (\cat{C}^\Omega \cdot a)(y)
    :=
    \{\alpha \cdot a \mid \alpha \in \Hom{C}{x}{y}\}.
    \]

    \item On morphisms: given $\beta \in \Hom{C}{y}{z}$,
    \[
    (\cat{C}^\Omega \cdot a)(\beta)
    :=
    \Omega(\beta)\big|_{(\cat{C}^\Omega \cdot a)(y)}.
    \]
\end{itemize}

Note that $\cat{C}^\Omega \cdot a$ is a $\cat{C}$-subset of $\Omega$.
\end{defi}

\begin{lem}[Theorem 4 in \cite{semisimple}]\label{orbita simple}
    Let $\Omega$ be a simple $\cat{C}$-set. Let $x\in\obj{C}$ be such that
$\Omega(x)\neq\varnothing$, and let $a\in\Omega(x)$. Then
\[
    \cat{C}^\Omega \cdot a = \Omega.
\]
\end{lem}
\begin{defi}
We define the \emph{socle} of a $\cat{C}$-set $\Omega$ as the union of all simple $\cat{C}$-subsets contained in it, that is,
\[
   soc(\Omega) := \bigcup \{\, S \mid S \text{ is a simple $\cat{C}$-set and } S \subseteq \Omega \,\}.
\]

A $\cat{C}$-set is said to be \emph{semisimple} if it is equal to its socle. Moreover, if $\Omega$ and $\Psi$ are isomorphic  $\cat{C}$-sets, then $soc(\Omega) \cong soc(\Psi)$.
\end{defi}
\begin{defi}
Let $\cat{C}^{Sim}$-set denote the full subcategory of semisimple $\cat{C}$-sets. The Grothendieck group $K_0(\cat{C}^{Sim}$-set, $\sqcup$), with addition induced by disjoint union, is denoted by $B^S(\cat{C})$ and is called the \emph{simple Burnside group} of $\cat{C}$.
\end{defi}

By construction, $(B^S(\cat{C}),+)$ is an abelian subgroup of $(B(\cat{C}),+)$. Since $B(\cat{C})$ decomposes as a direct sum of the Burnside rings of the connected subcategories of $\cat{C}$, the group $B^S(\cat{C})$ also decomposes as a direct sum of the simple Burnside rings of the connected subcategories. Consequently, from now on all categories considered will be finite and connected.
An interesting property of the product in $B(\cat{C})$ is that the product of classes of simple or semisimple $\cat{C}$-sets need not be simple or semisimple. This phenomenon is illustrated in the following example.

\begin{exam}
    Consider a monoid $M = \{e, \alpha, \beta\}$, whose operation is described in the following Cayley table:
    \[
        \begin{array}{c|cccc}
        \cdot & e & \alpha & \beta \\
        \hline
        e & e & \alpha & \beta \\
        \alpha & \alpha & \alpha & \alpha \\
        \beta & \beta & \beta & \beta 
        \end{array}
    \]

Let $\cat{C}$ be the delooping category of this monoid. We define a $\cat{C}$-set $\Omega$ by
\[
\Omega(\cdot) = \{a,b\},
\]
with action given by
\[
\begin{array}{ccc}
e \cdot a = a, & \alpha \cdot a = a, & \beta \cdot a = b, \\
e \cdot b = b, & \alpha \cdot b = a, & \beta \cdot b = b.
\end{array}
\]
A straightforward verification shows that $\Omega$ is a well-defined $\cat{C}$-set and that it is simple.

On the other hand, consider the $\cat{C}$-set $\Omega \times \Omega$, for which
\[
(\Omega \times \Omega)(\cdot) = \{(a,a), (a,b), (b,a), (b,b)\}.
\]
Note that the $\cat{C}$-set $\Lambda$, defined by assigning the subset $\{(a,a),(b,b)\}$, is a $\cat{C}$-subset of $\Omega \times \Omega$. This shows that $\Omega \times \Omega$ is not a simple $\cat{C}$-set.
\end{exam}
Let $\cat{C}$ be a finite connected category, and let $\Omega$ be a $\cat{C}$-set. We define the support of $\Omega$ by
\begin{align*}
    supp(\Omega) = \{\, x \in \cat{C}_0 \mid \Omega(x) \neq \emptyset \,\}.
\end{align*}

\begin{rem}
Let $\Omega$ and $\Psi$ be $\cat{C}$-sets. If there exists a natural transformation
$ 
\eta:\Omega\longrightarrow \Psi,
$ 
then $
\operatorname{Supp}(\Omega)\subseteq \operatorname{Supp}(\Psi).
$

\end{rem}

\begin{prop}
    Let $\Omega$ be a $\cat{C}$-set and let $S$ be a simple $\cat{C}$-set. If $\eta \colon \Omega \to S$ is a nonempty natural transformation. Then $\eta_x$ is surjetive for all $x\in \cat{C}_0$. 
\end{prop}
\begin{proof}
Note that it suffices to prove that $\eta_y$ is surjective for every $y\in \operatorname{Supp}(S)$. Since $\eta$ is nonempty, there exist an object $x\in\cat{C}_0$ and an element $a\in\Omega(x)$ such that $\eta_x(a)\in S(x)$. Let $y\in Supp(S)$. 
 Let $b\in S(y)$. Since $S$ is a simple $\cat{C}$-set, there exists a morphism $\beta\in Hom_{\cat{C}}(x,y)$ such that $ 
S(\beta)(\eta_x(a)) = b.
$
By the naturality of $\eta$, we have
\[
b = S(\beta)(\eta_x(a)) = \eta_y(\Omega(\beta)(a)).
\]
Thus, $\eta_y$ is surjective for every object $y\in\cat{C}_0$. 
\end{proof}

\begin{defi}
Let $\Omega$ be a $\cat{C}$-set. We define the \emph{order} of $\Omega$ by
\[
|\Omega| := \sum_{x\in\cat{C}_0} |\Omega(x)|,
\]
where $|\Omega(x)|$ denotes the cardinality of the set  $\Omega(x)$.
\end{defi}

\begin{prop}\label{| | <| |}
Let $\Omega$ be a $\cat{C}$-set and let $S$ be a simple $\cat{C}$-set. If $\eta \colon \Omega \to S$ is a natural transformation, then $|\Omega| \geq |S|$.
\end{prop}

\begin{proof}
Let $\eta \colon \Omega \to S$ be a natural transformation. If $\eta \neq \emptyset$, then for every $x \in \cat{C}_0$, the map $\eta_x \colon \Omega(x) \to S(x)$ is surjective. Hence $|\Omega(x)| \geq |S(x)|$ for all $x \in \cat{C}_0$, and therefore $|\Omega| \geq |S|$.
\end{proof}

\begin{cor}\label{| | =| |}
  Let $\Omega$ and $S$ be simple $\mathcal{C}$-sets such that $ |\Omega |=|S|$. If $\eta \colon \Omega \to S$ is a natural transformation, then $\eta$ is an isomorphism.
\end{cor}
\begin{proof}
By Proposition~\ref{| | <| |}, we have $|S(x)| \leq |\Omega(x)| $
for every $x \in \mathcal{C}_0$. Since $|\Omega| = |S|$, it follows that $|S(x)| = |\Omega(x)| $
for every $x \in \mathcal{C}_0$. Therefore, each map $\eta_x \colon \Omega(x) \to S(x)$
is a bijection. Hence, $\eta$ is a natural isomorphism.
\end{proof}

\begin{cor}
Let $S$ be a simple $\cat{C}$-set, let $x \in supp(S)$, and let $a \in S(x)$. Consider the natural transformation $  \eta \colon Hom_{\cat{C}}(x, -) \longrightarrow S$, defined by
\begin{align*}
    \eta_y \colon Hom_{\cat{C}}(x, y) &\longrightarrow S(y), \\
    \alpha &\longmapsto S(\alpha)(a).
\end{align*}
Then $\eta_y$ is surjective for all $y \in \cat{C}_0$.
\end{cor}

\begin{defi}
Let $\Omega$ be a $\cat{C}$-set. A family of relations 
\[
    \sim := \{\, \sim_x \,\}_{x \in \cat{C}_0}
\]
is called a \emph{congruence} on $\Omega$ if, for each $x \in \cat{C}_0$, the relation $\sim_x$ is an equivalence relation on $\Omega(x)$, and for every morphism $\alpha \in Hom_{\cat{C}}(x,y)$ and all $a,b \in \Omega(x)$ with $a \sim_x b$, one has
\[
    \Omega(\alpha)(a) \sim_y \Omega(\alpha)(b).
\]
\end{defi}

\begin{defi}[The quotient $\Omega/\sim$]
Let $\Omega$ be a $\cat{C}$-set and let $\sim = \{\, \sim_x \,\}_{x \in \cat{C}_0}$ be a congruence on $\Omega$. We define the quotient $\Omega/\sim$ as follows:
\begin{itemize}
    \item On objects:
    \[
        (\Omega/\sim)(x) := \Omega(x)/\sim_x, \quad \text{for all } x \in \cat{C}_0.
    \]
    
    \item On morphisms: for every morphism $\alpha \colon x \to y$ in $\cat{C}$,
    \[
        (\Omega/\sim)(\alpha) \colon \Omega(x)/\sim_x \longrightarrow \Omega(y)/\sim_y, 
        \quad [a] \longmapsto [\Omega(\alpha)(a)].
    \]
\end{itemize}
This defines a $\cat{C}$-set. The canonical projection $\pi_{\sim} \colon \Omega \to \Omega/\sim$ is a natural transformation.
\end{defi}
\begin{thm}[The universal property of the quotient]
Let $\Omega$ and $\Psi$ be $\cat{C}$-sets, let $\eta \colon \Omega \to \Psi$ be a natural transformation, and let $\sim = \{\, \sim_x \,\}_{x \in \cat{C}_0}$ be a congruence on $\Omega$. Then the quotient and the canonical morphism $\pi_{\sim} \colon \Omega \to \Omega/\sim$ satisfy the following universal property:
there exists a unique natural transformation
\[
    \overline{\eta} \colon \Omega/\sim \longrightarrow \Psi
\]
such that
\[
    \eta = \overline{\eta} \circ \pi_{\sim}
\]
if and only if, for all $x \in \obj{C}$ and all $a,b \in \Omega(x)$,
\[
    a \sim_x b \;\Rightarrow\; \eta_x(a) = \eta_x(b).
\]
\end{thm}
This can be interpreted as if the following diagram commutes. 
$$\xymatrix{\Omega \ar[rr]^{\pi_{\sim}}\ar[d]_{\eta} & & \Omega/\sim  \ar[lld]^{\exists ! \overline{\eta}}  \\
\Psi} $$

\begin{proof}
First, suppose that there exists a natural transformation $ \overline{\eta} \colon \Omega/\sim \longrightarrow \Psi$ such that $\eta = \overline{\eta} \circ \pi_{\sim}$. Let $x \in \obj{C}$ and let $a,b \in \Omega(x)$ be such that $a \sim_x b$. Then
\[
    {\pi_{\sim}}_x(a) = [a] = [b] = {\pi_{\sim}}_x(b).
\]
By hypothesis, $\eta_x = \overline{\eta}_x \circ {\pi_{\sim}}_x$, hence
\[
    \eta_x(a) = \overline{\eta}_x({\pi_{\sim}}_x(a)) = \overline{\eta}_x({\pi_{\sim}}_x(b)) = \eta_x(b).
\]

Conversely, define a natural transformation $\overline{\eta} \colon \Omega/\sim \to \Psi$ by
\[
    \overline{\eta}_x([a]) := \eta_x(a), \quad \text{for all } x \in \obj{C},\; [a] \in \Omega(x)/\sim_x.
\]
This map is well defined: if $[a] = [b]$, then $a \sim_x b$, and hence $\eta_x(a) = \eta_x(b)$, so
\[
    \overline{\eta}_x([a]) = \overline{\eta}_x([b]).
\]

By construction, $\eta = \overline{\eta} \circ \pi_{\sim}$. 

To prove uniqueness, let $\kappa \colon \Omega/\sim \to \Psi$ be another natural transformation such that $\eta = \kappa \circ \pi_{\sim}$. Then, for all $x \in \obj{C}$ and $a \in \Omega(x)$,
\[
    \overline{\eta}_x({\pi_{\sim}}_x(a)) = \eta_x(a) = \kappa_x({\pi_{\sim}}_x(a)).
\]
Since ${\pi_{\sim}}_x$ is surjective, it follows that $\overline{\eta}_x = \kappa_x$, and hence $\overline{\eta} = \kappa$.
\end{proof}

Let $S$ be a simple $\cat{C}$-set, let $x \in \cat{C}_0$, and let $a \in S(x)$. For any $y \in \cat{C}_0$, we define an equivalence relation $\sim_y^S$ on $\mathrm{Hom}_{\cat{C}}(x,y)$ as follows: for $\alpha, \beta \in \mathrm{Hom}_{\cat{C}}(x,y)$, we declare
\[
\alpha \sim_y^S \beta \quad \text{if and only if} \quad S(\alpha)(a) = S(\beta)(a).
\]
Note that for every $f \in \mathrm{Hom}_{\cat{C}}(y,z)$, if $\alpha \sim_y^S \beta$ in $\mathrm{Hom}_{\cat{C}}(x,y)$, then
$ f \circ \alpha \sim_z^S f \circ \beta$ 
in $\mathrm{Hom}_{\cat{C}}(x,z)$.
By the previous observation, the relation $\sim^S := \{\sim_y^S\}_{y \in \cat{C}_0}$ is a congruence on $\mathrm{Hom}_{\cat{C}}(x,-)$. Therefore, the following functor is well defined.

\begin{defi}
Let $S$ be a simple $\cat{C}$-set, let $x\in \operatorname{Supp}(S)$, and let $a\in S(x)$. We define the $\cat{C}$-set $\Omega_S^x$ by
\[
\Omega^x_S(y) := \mathrm{Hom}_{\cat{C}}(x,y)/\sim_y^S,
\quad \text{for each } y \in \cat{C}_0.
\]
\end{defi}
Moreover, the natural morphism of $\cat{C}$-sets
$ 
\eta \colon \mathrm{Hom}_{\cat{C}}(x,-)/\sim^S \longrightarrow S
$
is defined on each component $y \in \cat{C}_0$ by
\begin{align*}
    \eta_y \colon Hom_{\cat{C}}(x, y)/\sim^S_y &\longrightarrow S(y), \\
    [\alpha] &\longmapsto S(\alpha)(a),
\end{align*}
Moreover, $\eta_y$ is a bijection for all $y \in \cat{C}_0$.

\begin{thm} \label{simple ang quotinet}
 The functor \(\Omega_S^x := \mathrm{Hom}_{\mathcal{C}}(x,{-})/\sim\) is simple if and only if the following holds:
1. For all \(f \in \mathrm{Hom}_{\mathcal{C}}(y,z)\), whenever \(\alpha \sim_y \beta\) in \(\mathrm{Hom}_{\mathcal{C}}(x,y)\), then
\[
f \circ \alpha \sim_z f \circ \beta \quad \text{in } \mathrm{Hom}_{\mathcal{C}}(x,z).
\]

2. For all \([\alpha] \in \Omega_S^x(y)\) and \([\beta] \in \Omega_S^x(z)\), there exists \(l \in \mathrm{Hom}_{\mathcal{C}}(y,z)\) such that
\[
[l \circ \alpha] = [\beta].
\]

\end{thm}
\begin{proof}
The result follows from the fact that the second condition is equivalent to $\Omega_S^x$ consisting of a single orbit.
\end{proof}
\begin{cor}
There exist finitely many isomorphism classes of simple $\cat{C}$-sets.
\end{cor}
\begin{proof}
By Theorem~\ref{simple ang quotinet}, every simple $\cat{C}$-set is isomorphic to a quotient of \- 
$ Hom_{\cat{C}}(x,-)$ for some $x \in \cat{C}_0$. Since $\cat{C}_0$ is finite and each set $Hom_{\cat{C}}(x,y)$ is finite for all $y \in \cat{C}_0$, it follows that $Hom_{\cat{C}}(x,-)$ admits only finitely many quotients. Hence, there are only finitely many isomorphism classes of simple $\cat{C}$-sets.
\end{proof}

\begin{prop} \label{nat simple to soc}
Let $S$ be a simple $\cat{C}$-set, let $\Omega$ be a $\cat{C}$-set, and let $\eta \colon S \to \Omega$ be a natural transformation. Then $\eta$ factors through $soc(\Omega)$.
\[
\begin{tikzcd}
S \arrow[r,"\eta"] \arrow[dr,dashed,"\eta"] &
\Omega \\
&
\operatorname{soc}(\Omega) \arrow[u,hook]
\end{tikzcd}
\]
\end{prop}

\begin{rem}
By Lemma~\ref{orbita simple}, the union of simple $\cat{C}$-subsets of $\Omega$ is in fact a disjoint union of simple $\cat{C}$-subsets of $\Omega$.
\end{rem}

\begin{lem} \label{soc}
Let $\Omega$ and $\Lambda$ be $\cat{C}$-sets. Then $ soc( \Omega \times \Lambda) = soc\bigl(soc(\Omega) \times soc(\Lambda)\bigr). $
\end{lem}

\begin{proof}
First note that $soc(\Omega) \subseteq \Omega$ and $soc(\Lambda) \subseteq \Lambda$, hence
$ 
    soc(\Omega) \times soc(\Lambda) \subseteq \Omega \times \Lambda.
$
Taking socles, we obtain
$ 
    soc\bigl(soc(\Omega) \times soc(\Lambda)\bigr) \subseteq soc(\Omega \times \Lambda).
$

For the reverse inclusion, let $S \subseteq \Omega \times \Lambda$ be a simple $\cat{C}$-subset. Consider the projections
\[
    \pi_1 \colon \Omega \times \Lambda \to \Omega, \qquad
    \pi_2 \colon \Omega \times \Lambda \to \Lambda.
\]
Then $\pi_1(S)$ and $\pi_2(S)$ are $\cat{C}$-subsets of $\Omega$ and $\Lambda$, respectively. Since $S$ is simple, it follows that $\pi_1(S)$ and $\pi_2(S)$ are simple. As $S \neq \emptyset$, both are nonempty, hence simple. Therefore,
\[
    \pi_1(S) \subseteq soc(\Omega), \qquad \pi_2(S) \subseteq soc(\Lambda).
\]
It follows that
\[
    S \subseteq soc(\Omega) \times soc(\Lambda),
\]
and hence
\[
    soc(\Omega \times \Lambda) \subseteq soc\bigl(soc(\Omega) \times soc(\Lambda)\bigr).
\]

This completes the proof.
\end{proof}

\begin{defi}
Let $S$ and $T$ be simple $\cat{C}$-sets. We define 
\[
    [S] \star [T] := [soc(S \times T)],
\]
where $[S]$ and $[T]$ denote the isomorphism classes of $S$ and $T$, respectively. By Lemma~\ref{soc}, this product is well defined and is distributive with respect to disjoint union.
\end{defi}
\begin{prop}
The abelian group $(B^{S}(\cat{C}),+)$, endowed with the product $\star$, is a commutative ring with identity.
\end{prop}
\begin{proof}
Let $1_\cat{C} \colon \cat{C} \to \set$ be the functor defined by $1_\cat{C}(x) = \{\, \cdot \,\}$ (the unit of $B(\cat{C})$). We define
\[
    1_{B^S(\cat{C})} := soc(1_\cat{C}).
\]
Then $1_{B^S(\cat{C})}$ is the unit of $B^S(\cat{C})$. Indeed, for any simple $\cat{C}$-set $T$, we have
\begin{align*}
    [T] \star [1_{B^S(\cat{C})}]
    &= [soc(T \times 1_{B^S(\cat{C})})] \\
    &= [soc(T \times soc(1_\cat{C}))] \\
    &= [soc(soc(T) \times soc(1_\cat{C}))].
\end{align*}
By Lemma~\ref{soc}, this is equal to
$[soc(T \times 1_\cat{C})] = [soc(T)] = [T].$
By definition of the socle, the socle of a disjoint union is the disjoint union of the socles. Using also the distributivity of $\times$ over $\sqcup$, for any simple $\cat{C}$-sets $T, R_1, \dots, R_n$, we obtain
\begin{align*}
    [T] \star \left( \sum_{i=1}^n [R_i] \right)
    &= [soc\bigl(T \times \bigsqcup_{i=1}^n R_i\bigr)] \\
    &= [soc\bigl(\bigsqcup_{i=1}^n (T \times R_i)\bigr)] \\
    &= \sum_{i=1}^n [soc(T \times R_i)] \\
    &= \sum_{i=1}^n [T] \star [R_i].
\end{align*}

Finally, the associativity of $\star$ follows from the associativity of the Cartesian product $\times$.
\end{proof}
\begin{prop}
The map induced  by 
\begin{align*}
   Soc \colon B(\cat{C}) &\longrightarrow B^S(\cat{C}),\\
    [\Omega] &\longmapsto [soc(\Omega)],
\end{align*}
is a ring homomorphism.
\end{prop}

\begin{proof}
We first show that $soc$ preserves addition. Let $\Omega$ and $\Psi$ be $\cat{C}$-sets. Since the socle of a disjoint union is the disjoint union of the socles, we have
\[
   soc(\Omega \sqcup \Psi) =soc(\Omega) \sqcup soc(\Psi).
\]
Therefore,
\[
   Soc([\Omega] + [\Psi]) = [soc(\Omega \sqcup \Psi)] = [soc(\Omega)] + [soc(\Psi)]=soc([\Omega]) + Soc([\Psi]). 
\]

Next, we show that $soc$ preserves multiplication. By definition of the product,
\[
    [\Omega] \cdot [\Psi] = [\Omega \times \Psi],
\]
and by Lemma~\ref{soc}, we have
\[
   soc(\Omega \times \Psi) =soc(soc(\Omega) \times soc(\Psi)).
\]
Thus,
\begin{align*}
     Soc([\Omega] \star [\Psi]) &= [soc(\Omega \times \Psi)]
    = [soc(soc(\Omega) \times soc(\Psi))]\\&
    = [soc(\Omega)] \star [soc(\Psi)]=Soc([\Omega] )\star Soc([\Psi]).
\end{align*}

Finally, $Soc$ preserves the identity. and $soc(1_{\cat{C}})= 1_{B^S(\cat{C})}$.  Hence,  $Soc$ is a unital ring homomorphism.
\end{proof}

\section{The marks of the ring $B^S(\cat{C})$}\label{marks}
In this section, we describe the ring homomorphisms from $B^S(\cat{C})$ to $\mathbb{Z}$. The study of these homomorphisms in the case of the Burnside ring of a finite group has led to several remarkable results. For instance, groups with isomorphic Burnside rings have the same sublattice of solvable normal subgroups and the same spectrum (see~\cite{irreducible}). Recall that if $F$ and $G$ are functors from the category $\cat{C}$ to a category $\cat{D}$, then $Hom_{\cat{C}\text{-set}}(F,G)$ denotes the set of natural transformations from $F$ to $G$, and $|Hom_{\cat{C}\text{-set}}(F,G)|$ denotes its cardinality.

\begin{defi}
Let $S$ be a simple $\cat{C}$-set. The \emph{mark} of $S$ is the function
\begin{align*}
    \varphi_S \colon B^S(\cat{C}) &\longrightarrow \mathbb{Z},\\
    [\Omega] &\longmapsto |Hom_{\cat{C}\text{-set}}(S,\Omega)|.
\end{align*}
This function is well defined: if $\Omega \cong \Psi$, then $Hom_{\cat{C}\text{-set}}(S,\Omega) \cong Hom_{\cat{C}\text{-set}}(S,\psi)$.
\end{defi}
\begin{prop}
Let $S$ be a simple $\cat{C}$-set. Then the mark of $S$ is a ring homomorphism.
\end{prop}

\begin{proof}
Note that $B^S(\cat{C})$ is a subgroup of $B(\cat{C})$. then $\varphi_S$ is  abelian group homomorphism. By Proposition~\ref{nat simple to soc}, we have
\[
    |Hom_{\cat{C}\text{-set}}(S,\Omega)| = |Hom_{\cat{C}\text{-set}}(S,\mathrm{soc}(\Omega))|.
\]

Note that $1_{\cat{C}}(x)=\{ \cdot \}$ for all $x\in \cat{C}_0$. Hence,
\[
   1= |Hom_{\cat{C}\text{-set}}(S,1_{B^S(\cat{C})})| =| Hom_{\cat{C}\text{-set}}(S,soc(1_{\cat{C}}) )| = \varphi_S(1_{B^S(\cat{C})}),
\]
so $\varphi_S$ preserves the identity. The proof of additivity is identical to that of Proposition~\ref{mark}. We now prove multiplicativity. By definition,
\[
[R] \star [T] =[\mathrm{soc}(\mathrm{soc}(R)\times \mathrm{soc}(T))]= [\mathrm{soc}(R \times T)],
\]
and therefore
\[
\bigl|Hom_{\cat{C}\text{-set}}(S, R \star T)\bigr|
= \bigl|Hom_{\cat{C}\text{-set}}(S, \mathrm{soc}(R \times T))\bigr|.
\]
By Proposition~\ref{nat simple to soc},
\[
\bigl|Hom_{\cat{C}\text{-set}}(S, \mathrm{soc}(R \times T))\bigr|
= \bigl|Hom_{\cat{C}\text{-set}}(S, R \times T)\bigr|.
\]
Moreover,
\[
Hom_{\cat{C}\text{-set}}(S, R \times T)
\cong Hom_{\cat{C}\text{-set}}(S,R) \times Hom_{\cat{C}\text{-set}}(S,T),
\]
hence
\[
\bigl|Hom_{\cat{C}\text{-set}}(S, R \times T)\bigr|
= \bigl|Hom_{\cat{C}\text{-set}}(S,R)\bigr| \cdot \bigl|Hom_{\cat{C}\text{-set}}(S,T)\bigr|.
\]
Thus,$\varphi_S([R] \star [T])
= \varphi_S([R]) \cdot \varphi_S([T]).$
Therefore, $\varphi_S$ preserves addition, multiplication, and the identity, and hence it is a ring homomorphism.
\end{proof}
\begin{thm}\label{funtion and marcks}
Let $R$ be a ring of integers and let   $\varphi \colon B^S(\cat{C}) \to R$  be a ring homomorphism. Then there exists a simple $\cat{C}$-set $S$ such that
\[
    \varphi = \varphi_S\cdot 1_R.
\]
\end{thm}
\begin{proof}
Let $\varphi \neq 0$ be a ring homomorphism, which implies that $\varphi(1_{B^S(\mathcal{C})}) = 1_R$. It follows that $\varphi([S]) \neq 0$ for some simple $\mathcal{C}$-set $S$. Since there are only finitely many isomorphism classes of simple $\mathcal{C}$-sets, we can choose $S$ to be a simple $\mathcal{C}$-set of maximal cardinality $|S|$ satisfying this property.
Let $T$ be any simple $\mathcal{C}$-set; then $\varphi([S] \star [T]) = \varphi([S])\varphi([T])$. By the definition of the product, we have $[S] \star [T] = [\mathrm{soc}(S \times T)]$.  Write 
\[
\mathrm{soc}(S \times T) = \bigsqcup_{i=1}^n U_i,
\]
where each $U_i$ is a simple $\mathcal{C}$-set. Since each $U_i \leq S \times T$, consider the projection map $\pi_S \colon U_i \longrightarrow S$. Because $U_i$ is simple $\mathcal{C}$-set  non-empty, $\pi_S$ cannot be the trivial empty map; hence, it must be surjective, which implies that $|S| \leq |U_i|$. By the maximality of $|S|$, we must have $\varphi(U_i)=0 $ or $|U_i| = |S|$, and therefore $U_i \cong S$. 

Thus, applying $\varphi$ to the product yields
\[
\varphi([S] \star [T]) = \sum_{i=1}^n \varphi([U_i]) = |\{\, i \mid U_i \cong S \,\}| \cdot \varphi([S]).
\]
Since $\varphi([S] \star [T]) = \varphi([S])\varphi([T])$ and $\varphi([S]) \neq 0$, we can cancel $\varphi([S])$ from both sides to conclude that
\[
\varphi([T]) = |\{\, i \mid U_i \cong S \,\}|.
\]
On the other hand, considering the specific homomorphism $\varphi_S$, we have
\[
\varphi_S([S] \star [T]) = \sum_{i=1}^n \varphi_S([U_i]).
\]
Since $S$ is simple and $U_i\to S$, the value of $\varphi_S([U_i])$ is given by
\[
\varphi_S([U_i]) =
\begin{cases}
\varphi_S([S]) & \text{if } U_i \cong S, \\
0 & \text{otherwise}.
\end{cases}
\]
Therefore, we obtain
\[
\varphi_S([S])\varphi_S([T]) = |\{\, i \mid U_i \cong S \,\}| \cdot \varphi_S([S]).
\]
Since $\varphi_S([S]) \neq 0$, it follows that $\varphi_S([T]) = |\{\, i \mid U_i \cong S \,\}|$. Consequently, we conclude that $\varphi([T]) = \varphi_S([T])$ for any simple $\mathcal{C}$-set $T$. Given that the isomorphism classes of simple $\mathcal{C}$-sets form a generating set (a basis) for the simple Burnside ring, it follows by linearity that $\varphi = \varphi_S$.

\end{proof}
\begin{lem}
Let $S$ be a simple $\cat{C}$-set, and let $T$ be a
$\cat{C}$-set. Then, for every $x\in\operatorname{supp}(S)$,
there exists an injective map
\[
\theta\colon
\operatorname{Hom}_{\cat{C}\text{-set}}(S,T)
\hookrightarrow T(x).
\]
\end{lem}
\begin{proof}
Let $x\in\operatorname{supp}(S)$ and choose $a\in S(x)$.

Define
\begin{align*}
\theta_a\colon
\operatorname{Hom}_{\cat{C}\text{-set}}(S,T)
&\longrightarrow T(x),\\
\beta&\longmapsto\beta_x(a).
\end{align*}

We claim that $\theta_a$ is injective.

Let $\alpha,\beta\in
\operatorname{Hom}_{\cat{C}\text{-set}}(S,T)$
such that
\[
\theta_a(\alpha)=\theta_a(\beta).
\]
Thus,
\[
\alpha_x(a)=\beta_x(a).
\]

Let $y\in\cat{C}_0$ and $b\in S(y)$.
Since $S$ is simple, the sub-$\cat{C}$-set
generated by $a$ is $S$. Hence, there exists
a morphism
\[
\psi\in\operatorname{Hom}_{\cat{C}}(x,y)
\]
such that
\[
b=S(\psi)(a).
\]

By naturality, we obtain
\begin{align*}
\beta_y(b)
&=\beta_y(S(\psi)(a))\\
&=T(\psi)(\beta_x(a))\\
&=T(\psi)(\alpha_x(a))\\
&=\alpha_y(S(\psi)(a))\\
&=\alpha_y(b).
\end{align*}

Therefore, $\alpha_y=\beta_y$ for every
$y\in\cat{C}_0$, and consequently
$\alpha=\beta$.

Hence, $\theta_a$ is injective.
\end{proof}

\section{The spectrum of the ring $B^S(\cat{C})$}\label{spec}
In this section, we give an explicit description of the prime spectrum of the ring $B^S(\cat{C})$. Our approach is based on the mark homomorphisms introduced in the previous sections, which allow us to characterize the prime ideals of $B^S(\cat{C})$.

\begin{defi}
Let $\cat{C}$ be a finite connected category. The \emph{spectrum} of the ring $B^S(\cat{C})$ is defined as
\[
\mathrm{Spec}(B^S(\cat{C})) := \{\, \mathfrak{p} \leq B^S(\cat{C}) \mid \mathfrak{p} \text{ is a prime ideal} \,\}.
\]
\end{defi}
\begin{nota}
Let $S$ be a simple $\cat{C}$-set. We set
\[
\mathfrak{p}_S := \ker(\varphi_S).
\]
Since $\mathbb{Z}$ is an integral domain, $\mathfrak{p}_S$ is a prime ideal of $B^S(\cat{C})$.

Let $p$ be a prime number. Consider the canonical projection
\[
\pi_p \colon \mathbb{Z} \longrightarrow \mathbb{Z}/p\mathbb{Z},
\]
and define the ring homomorphism
\[
\pi_p \circ \varphi_S \colon B^S(\cat{C}) \longrightarrow \mathbb{Z}/p\mathbb{Z}.
\]

Set
\[
\mathfrak{p}_{S,p} := \ker(\pi_p \circ \varphi_S).
\]

Since $\mathbb{Z}/p\mathbb{Z}$ is a field, it follows that $\mathfrak{p}_{S,p}$ is a maximal ideal of $B^S(\cat{C})$.
\end{nota}
\begin{thm}
We have
\[
\mathrm{Spec}(B^S(\cat{C})) = \{\, \mathfrak{p}_S,\ \mathfrak{p}_{S,p} \mid S \text{ simple $\cat{C}$-set} ,\ p \text{ prime} \,\}.
\]
\end{thm}

\begin{proof}
Let $\mathfrak{p} \in \mathrm{Spec}(B^S(\cat{C}))$ .  Then the quotient $ R:=B^S(\cat{C})/\mathfrak{p}$
is an integral domain.
Hence, the canonical projection 
\[
f \colon B^S(\cat{C}) \longrightarrow R,
\]
is a ring homomorphism. 
By the Theorem $\ref{funtion and marcks}$,  we have that $ f = \varphi_S \circ 1_R$ 
for some simple $\cat{C}$-set $S$.

If $\mathrm{char}(B^S(\cat{C})/\mathfrak{p}) = 0$, then
$ \mathfrak{p} = \ker(f) = \ker(\varphi_S) = \mathfrak{p}_S.$

If $\mathrm{char}(B^S(\cat{C})/\mathfrak{p}) = p > 0$, then $p$ is prime and
\[
f \colon B^S(\cat{C}) \longrightarrow \mathbb{Z} \longrightarrow \mathbb{Z}/p\mathbb{Z},
\]
so that
\[
\mathfrak{p} = \ker(f) = \ker(\pi_p \circ \varphi_S) = \mathfrak{p}_{S,p}.
\]

Therefore, $\mathfrak{p} \in \{\, \mathfrak{p}_S,\ \mathfrak{p}_{S,p} \,\}.
 $
\end{proof}

\begin{lem}
Let $S$ and $T$ be simple $\cat{C}$-sets. Then
\[
\mathfrak{p}_S = \mathfrak{p}_T \quad \Longleftrightarrow \quad S \cong T.
\]
\end{lem}

\begin{proof}
$(\Rightarrow)$ Suppose that $\mathfrak{p}_S=\mathfrak{p}_T$. Then
$\ker(\varphi_S)=\ker(\varphi_T)$.
Since $T$ is not an element of  $\ker(\varphi_T)$, it follows that $T\notin \ker(\varphi_S)$, and hence
$\varphi_S(T)\neq 0$.
Therefore, there exists a nonzero natural transformation
$\eta:S\to T$.
By Proposition \ref{| | <| |}, this implies that $|T|\leq |S|$. By symmetry, we also obtain $|S|\leq |T|$. Hence $|S|=|T|$. Since $S$ and $T$ are simple $\cat{C}$-sets, it follows that $\eta$ is an isomorphism. Therefore, $S\cong T$.

$(\Leftarrow)$ Conversely, if $S\cong T$, then clearly $\varphi_S=\varphi_T$. Therefore,
$\mathfrak{p}_S=\mathfrak{p}_T$.
\end{proof}

\section{Decomposition of  $B^S(\cat{C})$}\label{descompositon}

In this section, we study a decomposition of the ring $B^S(\cat{C})$. To do so, we decompose the category $\cat{C}$ into subcategories, which preserve the information of simple $\cat{C}$-sets under restriction.

\begin{lem} \label{supp si sj}
Let $1_{\cat{C}} \in B(\cat{C})$ be such that $soc(1_\cat{C}) = \sum_{i=1}^n S_i$,
where each $S_i$ is simple. Then
\[
    supp(S_i) \cap supp(S_j) = \emptyset, \quad \text{for } i \neq j.
\]
\end{lem}
\begin{proof}
 We note that each $S_i$ is a subfunctor of $1_{\cat{C}}$. Hence, for every
$x \in supp(S_i)$, we have $S_i(x) = \{\, \cdot \,\}$. Since
\[
    soc(1_\cat{C})(x) = \sum_{i=1}^n S_i(x) = \{\, \cdot \,\},
\]
it follows that $S_i \neq S_j$ for $i \neq j$.

Now, suppose that $x \in supp(S_i) \cap supp(S_j)$ with $i \neq j$.
Then
\[
    S_i(x) = \{\, \cdot \,\} = S_j(x).
\]
For any $y \in \cat{C}_0$, we have
\begin{align*}
   S_i(y)
    = S_i\bigl(Hom_{\cat{C}}(x,y)\bigr)(\cdot)
    = S_j\bigl(Hom_{\cat{C}}(x,y)\bigr)(\cdot)
    = S_j(y) 
\end{align*}
Therefore $S_i = S_j$, which is a contradiction.
\end{proof}
\begin{defi}
Let $1_\cat{C} \in B(\cat{C})$ be the unit $\cat{C}$-set. We call the simple $\cat{C}$-sets that appear in the decomposition of $soc(1_{\cat{C}})$ the \emph{minimal simple $\cat{C}$-sets}.
\end{defi}

\begin{defi} \label{cat ci}
Let $1_\cat{C} \in B(\cat{C})$ and suppose that $soc(1_\cat{C}) = \sum_{i=1}^n S_i.$ 
The elements of $\cat{C}_0 \setminus \bigsqcup_{i=1}^n supp(S_i)$ are called  \emph{pseudo-sources}.  
For each $i$, we denote by $\cat{C}_i$ the full subcategory of $\cat{C}$
whose object set is $ (\cat{C}_i)_0 = supp(S_i).$
\end{defi}
\begin{cor}
Let $x \in (\cat{C}_i)_0$ and $y \in (\cat{C}_j)_0$ with $i \neq j$. Then
$  Hom_{\cat{C}}(x,y) = \emptyset.$
\end{cor}

\begin{proof}
This follows from Lemma~\ref{supp si sj}.
\end{proof}

Let $T$ be a $\cat{C}$-set. Then there exists a morphism $\eta \colon T \to 1_{\cat{C}}$. If $T$ is simple. Thus, $\eta$ factors through $soc(1_{\cat{C}})$.
 \begin{rem}
Let $T$ be a simple $\cat{C}$-set. Then there exists a minimal simple $\cat{C}$-set $S_i$ such that there is a non-empty natural transformation $\eta_1:T\longrightarrow S_i.$
 \end{rem}
 \begin{proof}
Note that for every $x \notin supp(S_i)$, the morphism
\[
\eta_{1,x} : T(x) \to S_i(x)
\]
has codomain $\emptyset$, since $S_i(x)=\emptyset$. Hence $T(x)=\emptyset$.

Suppose there exists $j \neq i$ and a morphism $\eta_2 : T \to S_j$.
For every $x \in supp(S_i)$, which is disjoint from $supp(S_j)$, we have
$S_j(x)=\emptyset$. Therefore, the component
\[
\eta_{2,x} : T(x) \to S_j(x)
\]
is the empty function, and consequently $T(x)=\emptyset$ for all
$x \in supp(S_i)$.

Thus, $T(x)=\emptyset$ for all $x \in \cat{C}_0$, which implies $T=\emptyset$.
This is a contradiction.
\end{proof}
Let $T$ be a $\cat{C}$-set then there exits $\eta : T\to 1$, if $T$ is a simple,   then  $\eta : T\to soc(1_\cat{C})$. In consequence $Im(\eta (S_T))\leq S_i$
Then we have for every $T$ simple $\cat{C}$-set, there exist only $S_i$  minimal such that $T\to S_i$, then $supp(T)\subseteq supp(S_i)$.

 \begin{thm}\label{descom de B^S}
 Let $1_\cat{C}\in B(\cat{C})$ such that $soc(1_\cat{C}) = \sum_{i=1}^n S_i.$  Then 
 \begin{align*}
 B^S(\cat{C})\cong\prod_{i=1}^n B^s({\cat{C}_i})
 \end{align*}
 where each $\cat{C}_i$ is the full subcategory of $\cat{C}$ defined in Definition~\ref{cat ci}.
 \end{thm}
\begin{proof}
    Let $S_i$ and $S_j$ be minimal  simple $\cat{C}$-set with $i \neq j$. We note that $ S_i \times S_j = \emptyset,$
and therefore $ [S_i]\cdot [S_j] = [soc(S_i \times S_j)] =[\emptyset].
$ 
On the other hand, for every $x \in supp(S_i)$ we have
\[
(S_i \times S_i)(x)
= \{\cdot\} \times \{\cdot\}
\cong \{\cdot\}
= S_i(x),
\]
and for every $x \notin supp(S_i)$,
\[
(S_i \times S_i)(x) = \emptyset = S_i(x).
\]
Hence, $soc(S_i \times S_i) \cong S_i,$ 
and thus $[S_i]\cdot [S_i] = [S_i].$
Since $ 1_{B^S(\cat{C})}=\sum_{i=1}^n [S_i]$, it follows that the elements
$[S_i]$, for $i=1,\dots,n$, are orthogonal primitive idempotents. Therefore,
\[
B^S(\cat{C}) = \bigoplus_{i=1}^n B^S(\cat{C})[S_i].
\]
We now prove that
\[
B^S(\cat{C})[S_i] = B^S(\cat{C}_i).
\]
Let $T$ be a simple $\cat{C}$-set. If $supp(T) \nsubseteq supp(S_i)$, then $T \times S_i = \emptyset.$
Otherwise, if $supp(T) \subseteq supp(S_i)$, we have $T \times S_i \cong T.$ It follows that $[T]\cdot [S_i] = [T],$
and hence $[T] \in B^S(\cat{C}_i)$.

\end{proof}

Let $S_i$ be a simple minimal $\cat{C}$-set. Then we have
\[
\cat{C}^{S_i} \cdot \{\cdot\}_x(y)
= S_i(Hom_{\cat{C}}(x,y))
= S_i(y),
\]
for all $x,y \in supp(S_i)$. Consequently,
\[
Hom_{\cat{C}_i}(x,y) \neq \emptyset
\quad \text{for all } x,y \in (\cat{C}_i)_0.
\]

By the decomposition of $B^S(\cat{C})$ and the condition on the morphism
sets $Hom_{\cat{C}_i}(-,-)$, we are led to the following definition.
\begin{defi}
A finite connected category $\cat{C}$ is said to be \emph{strongly connected} if
$Hom_{\cat{C}}(x,y)\neq\emptyset$
for all $x,y\in\cat{C}_0$.
\end{defi}
The categories $\cat{C}_i$ introduced in Definition~\ref{cat ci}, and 
groupoids, are  strongly connected categories. We now introduce another example of a strongly connected category that will be used extensively in this article.
\begin{exam}\label{the example}
Let 
$ M=\{1,\rho,e,\tau\} $ 
be a monoid with identity element $1$, whose multiplication is determined by the following relations:
\[
\rho^2=1, \qquad ex=e, \qquad \tau x=\tau \quad \text{for all } x\in M.
\]
Equivalently, the multiplication in $M$ is given by the following the  Cayley table:
\[
\begin{array}{c|cccc}
\cdot & 1 & \rho & e & \tau \\
\hline
1     & 1 & \rho & e & \tau \\
\rho  & \rho & 1 & \tau & e \\
e     & e & e & e & e \\
\tau  & \tau & \tau & \tau & \tau
\end{array}
\]
Let $\cat{D}$ be the category defined as follows.

\begin{itemize}

\item The set of objects of $\cat{D}$ is $ \cat{D}_0=\{a,b\}. $

\item For $x,y\in \{a,b\}$ we define the sets of morphisms as disjoint sets
\[
Hom_{\cat{D}}(x,y)=
\begin{cases}
\{(m,a,b)\mid m\in \{ e,\tau\} \} & \text{if } (x,y)=(a,b),\\[4pt]
\{(m,b,b)\mid m\in M\} & \text{if } (x,y)=(b,b),\\[4pt]
\{(e,a,a)\} & \text{if } (x,y)=(a,a),\\[4pt]
\{(e,b,a)\} & \text{if } (x,y)=(b,a).
\end{cases}
\]

\item The composition of morphisms is defined as follows.  
Let $m,n\in M$. Whenever the composition is defined, we set
\[
(m,y,z)\circ(n,x,y)=(mn,x,z),
\]
where $mn$ denotes the product in the monoid $M$. 

\item The identity morphisms are
\[
\mathrm{id}_a=(e,a,a), \qquad \mathrm{id}_b=(1,b,b).
\]
Since,  $(\tau,a,b)\circ (e,a,a)=(\tau,a,b)$,  $(e,a,b)\circ (e,a,a)=(e,a,b),$ and $(e,a,a)\circ (e,b,a)=(e,b,a)$
\end{itemize}

Associativity of composition follows from the associativity of the multiplication in $M$.
\end{exam}
\section{Strongly connected categories}\label{strongly connected}
If we want to study the Burnside ring of simple $\cat{C}$-sets, $B^S(\cat{C})$,  then by Theorem~\ref{descom de B^S}, it suffices to study the ring $B^S(\cat{C})$ for  strongly connected categories. Hence, from now on, all categories we consider will have strongly connected.

The first observation we make is that. The category  $\cat{C}$ is a strongly connected if and only if the $\cat{C}$-set $1_{\cat{C}}$ is simple. Furthermore, if $S$ is a simple $\cat{C}$-set, then $supp(S)$ is $\cat{C}_0$.
\begin{lem}
Let $\cat{C}$ be a strongly connected category. Then there exists a unique simple $\cat{C}$-set $M$, maximal with respect to the order, such that for every simple $\cat{C}$-set $S$ there exists a morphism $M \to S$.
\end{lem}
\begin{proof}
First, we show that the maximal simple $\cat{C}$-set is unique up to isomorphism.  Let $S_1$ and $S_2$ be simple $\mathcal{C}$-sets, each maximal with respect to $|S_i|$, and let $W$ be a simple $\mathcal{C}$-set such that $W \leq \mathrm{soc}(S_1 \times S_2)\leq S_1 \times S_2$. Then
\[
\begin{tikzcd}
& S_2 \\
W \arrow[ru] \arrow[r] \arrow[rd] & S_1 \times S_2 \arrow[u] \arrow[d] \\
& S_1
\end{tikzcd}
\]
Thus $W \to S_i$ is an isomorphism. Hence $S_1 \cong S_2 \cong W$.

Let $T$ be a simple $\mathcal{C}$-set, and let $M$ be a simple $\mathcal{C}$-set maximal with respect to $|M|$. Let $U \leq \mathrm{soc}(M \times T)$, a simple $\mathcal{C}$-set, we have the following diagram
\[
\begin{tikzcd}
& T \\
U \arrow[ru] \arrow[r] \arrow[rd] & M \times T \arrow[u] \arrow[d] \\
& M
\end{tikzcd}
\]
Tn particular, $|M| \leq |U|$. By maximality of $M$, it follows that $U \cong M$. Consequently, there exists a morphism $M \to T$.
    
\end{proof}
\begin{prop} \label{maximal in hom }
Let $M$ be a maximal simple $\cat{C}$-set and let $x \in \cat{C}_0$. Then
\[
\mathrm{soc}(\mathrm{Hom}_{\cat{C}}(x,-)) \cong \bigsqcup_{i=1}^n M
\]
for some $n \in \mathbb{N}$.
\end{prop}

\begin{proof}
Let $a \in M(x)$. Define a natural transformation
\[
\theta \colon \mathrm{Hom}_{\cat{C}}(x,-) \longrightarrow M
\]
by setting, for each $y \in \cat{C}_0$,
\[
\theta_y(\alpha)=M(\alpha)(a), \quad \text{for all } \alpha \in \mathrm{Hom}_{\cat{C}}(x,y).
\]

Let $S$ be a simple $\cat{C}$-subset of $\mathrm{Hom}_{\cat{C}}(x,-)$. Then the restriction $\theta|_S \colon S \to M$ is a nonzero morphism of $\cat{C}$-sets. By Proposition~\ref{| | <| |}, we have $|M| \leq |S|$. Since $M$ is maximal with respect to its cardinality, it follows that $|M| = |S|$, and hence $S \cong M$.
Therefore, every simple $\cat{C}$-subset of $\mathrm{Hom}_{\cat{C}}(x,-)$ is isomorphic to $M$. Since $|\mathrm{Hom}_{\cat{C}}(x,-)|$ is finite, it contains only finitely many simple $\cat{C}$-subsets. Thus,
\[
\mathrm{soc}(\mathrm{Hom}_{\cat{C}}(x,-)) \cong \bigsqcup_{i=1}^n M
\]
for some $n \in \mathbb{N}$.
\end{proof}

\begin{exam}
Continuing with Example~\ref{the example}, we compute the representable functors.

Let $\cat{D}$ be the category defined above. We compute the representable functors
\[
\mathrm{Hom}_{\cat{D}}(a,-), \qquad \mathrm{Hom}_{\cat{D}}(b,-) \colon \cat{D} \to \mathbf{Set}.
\]

\medskip

\textbf{The functor $\mathrm{Hom}_{\cat{D}}(a,-)$.}

By definition,
\[
\mathrm{Hom}_{\cat{D}}(a,-)(x)=\mathrm{Hom}_{\cat{D}}(a,x).
\]

Hence
\[
\mathrm{Hom}_{\cat{D}}(a,-)(a)=\mathrm{Hom}_{\cat{D}}(a,a)=\{(e,a,a)\},
\]
and
\[
\mathrm{Hom}_{\cat{D}}(a,-)(b)=\mathrm{Hom}_{\cat{D}}(a,b)
=\{(e,a,b),(\tau,a,b)\}.
\]

Observe that
\[
M\cdot \tau = M\cdot e = \{\tau,e\}.
\]

Moreover,
\[
(\tau,a,b)\circ (e,a,a)=(\tau,a,b),
\]
and
\[
(e,b,a)\circ (\tau,a,b)=(e,a,a), \qquad
(e,b,a)\circ (e,a,b)=(e,a,a).
\]

Let $T$ be a nonempty sub-$\cat{D}$-set of $\mathrm{Hom}_{\cat{D}}(a,-)$. Since
\[
\mathrm{Hom}_{\cat{D}}(a,a)=\{(e,a,a)\},
\]
we must have $(e,a,a)\in T(a)$.

Applying the morphisms $a\to b$, we obtain
\[
(\tau,a,b)\circ (e,a,a)=(\tau,a,b), \qquad
(e,a,b)\circ (e,a,a)=(e,a,b),
\]
hence
\[
(e,a,b),(\tau,a,b)\in T(b).
\]

Therefore $T=\mathrm{Hom}_{\cat{D}}(a,-)$. Consequently, $\mathrm{Hom}_{\cat{D}}(a,-)$ has no nontrivial sub-$\cat{D}$-sets and hence it is simple.

\medskip

\textbf{The functor $\mathrm{Hom}_{\cat{D}}(b,-)$.}

Similarly,
\[
\mathrm{Hom}_{\cat{D}}(b,-)(x)=\mathrm{Hom}_{\cat{D}}(b,x).
\]

Hence
\[
\mathrm{Hom}_{\cat{D}}(b,-)(a)=\mathrm{Hom}_{\cat{D}}(b,a)=\{(e,b,a)\},
\]
and
\[
\mathrm{Hom}_{\cat{D}}(b,-)(b)=\mathrm{Hom}_{\cat{D}}(b,b)
=\{(m,b,b)\mid m\in M\}.
\]

Observe that $M\cdot \tau = M\cdot e = \{\tau,e\}$. 

Consider the $\cat{D}$-subset generated by $(\tau,b,b)$ in $\mathrm{Hom}_{\cat{D}}(b,-)$. Since
\[
M\cdot \tau = \{\tau,e\},
\]
we obtain
\[
\cat{D}^{\mathrm{Hom}(b,-)}\cdot (\tau,b,b)
=
\cat{D}^{\mathrm{Hom}(b,-)}\cdot (e,b,b).
\]

This $\cat{D}$-subset is given by
\[
\cat{D}^{\mathrm{Hom}(b,-)}\cdot (\tau,b,b)(a)
=\{(e,b,a)\},
\]
and
\[
\cat{D}^{\mathrm{Hom}(b,-)}\cdot (\tau,b,b)(b)
=\{(\tau,b,b),(e,b,b)\}.
\]

Therefore, this defines a nontrivial proper sub-$\cat{D}$-set of $\mathrm{Hom}_{\cat{D}}(b,-)$. Consequently, $ 
\mathrm{Hom}_{\cat{D}}(b,-)
$ 
is not a simple $\cat{D}$-set. By Proposition~\ref{maximal in hom } the $\cat{D}$-set $\mathrm{Hom}_{\cat{D}}(a,-)$ is the maximal simple $\cat{D}$-set. Moreover, it is the unique nontrivial  simple $\cat{D}$-set, since one evaluation has one element while the other has two elements. We also have
\begin{align*}
    \mathrm{Hom}_{\cat{D}}(a,-)\cong \mathrm{Hom}_{\cat{D}}(a,-)\star \mathrm{Hom}_{\cat{D}}(a,-)
\end{align*}
since the functor generated by
\[
((e,a,a),(e,a,a))
\in
\mathrm{Hom}_{\cat{D}}(a,-)\times \mathrm{Hom}_{\cat{D}}(a,-)(a)
\]
is the unique simple $\cat{D}$-subset of
\[
\mathrm{Hom}_{\cat{D}}(a,-)\times \mathrm{Hom}_{\cat{D}}(a,-).
\]
Hence,
\[
\mathrm{Hom}_{\cat{D}}(a,-)\times \mathrm{Hom}_{\cat{D}}(a,-)
\cong
\mathrm{Hom}_{\cat{D}}(a,-).
\]

Therefore,
\[
B^S(D)\cong \mathbb{Z}[x]/\langle x^2-x \rangle .
\]
\end{exam}
\section*{Open Questions}

This article leaves several questions open for future research. One natural problem is to describe the connected components of the spectrum of $B^S(\cat{C})$ and to determine when two prime ideals $\mathfrak{p}_{S,p}$ and $\mathfrak{p}_{T,q}$ belong to the same connected component. Another interesting question is to understand what information about two categories can be deduced from the fact that they have isomorphic simple Burnside rings or the same marks. In the classical case of Burnside rings of finite groups, questions of this type have already been studied; for instance, Raggi-Cárdenas and Valero investigated groups with isomorphic Burnside rings in \cite{raggi2005groups}.

We would like to thank Peter Webb, who visited the CCM-UNAM in Morelia and posed several interesting questions suggesting possible directions for future research. Among them are the study of what happens when one replaces simple functors from a category $\cat{C}$ to the category of finite sets by simple functors taking values in a larger category containing finite sets, as well as the problem of computing the simple Burnside ring for more general categories $\cat{C}$.
\section*{Funding declaration}

J. Miguel Calder\'on was supported by SECIHTI with an EPM postdoctoral grant. Alberto G. Raggi-C\'ardenas and Itzel Rosas were partially supported by UNAM--PAPIIT research grant IN100226. Itzel Rosas was also partially supported by SECIHTI grant 4022250. Ram\'on H. Ruiz-Medina was supported by SECIHTI, postdoctoral fellowship I1200/111/2024.

\section*{Declarations}
Conflict of interest statement. We have no potential conflict of interest. 

Data availability statement. We declare that there is no availability of data and materials for this manuscript. 

Author contribution statement. All authors contributed to the manuscript equally.

\bibliographystyle{plain}
\bibliography{bibliography}

\end{document}